\documentclass [11pt,oneside,a4paper,mathscr]{amsart}

\usepackage{amscd,amsmath,amssymb,euscript}
\usepackage[frame,cmtip,curve,arrow,matrix,line,graph]{xy}
\usepackage{tikz-cd}
\usepackage{bigints}
\usetikzlibrary{matrix}
\usepackage{hyperref}
\usepackage{stmaryrd}

\title[KHAs as Hopf algebras]{K-theoretic Hall algebras of quivers with potential as Hopf algebras}
\author{Tudor P\u adurariu}
\address{Department of Mathematics, Columbia University, 
2990 Broadway, New York, NY 10027}
\email{tgp2109@columbia.edu}

\newtheorem{thm}{Theorem}[section]

\newtheorem{prop}[thm]{Proposition}

\theoremstyle{definition}

\newtheorem{thm*}[thm]{Theorem$^*$}

\newcommand{\comment}[1]{}

\renewcommand{\leq}{\leqslant}
\renewcommand{\geq}{\geqslant}

\newcommand{\X}{\mathcal{X}}

\begin{document}
\maketitle

\begin{abstract}
Preprojective K-theoretic Hall algebras (KHAs), particular cases of KHAs of quivers with potential, are conjecturally positive halves of the Okounkov--Smirnov affine quantum algebras. It is thus natural to ask whether KHAs of quivers with potential are halves of a quantum group.
For a symmetric quiver with potential satisfying a Kunneth-type condition,
we construct (positive and negative) extensions of its KHA 
which are bialgebras. In particular, there are bialgebra extensions of preprojective KHAs and one can construct their Drinfeld double algebra.
\end{abstract}

\section{Introduction}
\subsection{Hall algebras for quivers with potential} 
Let $Q=(I,E)$ be a symmetric quiver, let $W$ be a potential of $Q$, and let $d\in\mathbb{N}^I$ be a dimension vector of $Q$.
Denote by $\X(d)$ the stack of representations of $Q$ of dimension $d$ over $\mathbb{C}$ and by $\X(d)_0$ the (derived) zero locus of the regular function \[\text{Tr}\,W:\X(d)\to\mathbb{A}^1_\mathbb{C}.\] 
Let $D_{\text{sg}}\left(\X(d)_0\right)$
be the category of singularities. It is equivalent to the category of (coherent) matrix factorizations $\text{MF}(\X(d), W)$ for the regular function $\text{Tr}\,W$. 
The categorical Hall algebra is the category $\bigoplus_{d\in\mathbb{N}^I} D_{\text{sg}}\left(\X(d)_0\right)$ with multiplication
\begin{equation}\label{productHA}
m_{d,e}:=p_{d,e*}q_{d,e}^*: D_{\text{sg}}\left(\X(d)_0\right)\otimes D_{\text{sg}}\left(\X(e)_0\right)\to D_{\text{sg}}\left(\X(d+e)_0\right),
\end{equation}
where $q_{d,e}:\X(d,e)\to \X(d)\times \X(e)$ and $p_{d,e}:\X(d,e)\to\X(d+e)$ are the natural maps from the stack of extensions $\X(d,e)$. The K-theoretic Hall algebra (KHA) of $(Q,W)$ is $K_0$ of the categorical Hall algebra.

The above construction is suggested by Kontsevich--Soibelman in \cite{ks} where they studied its cohomological version, see \cite{P} for more details about the categorical and K-theoretic constructions.

\subsection{Preprojective Hall algebras}
For $Q$ a quiver, let $\mathfrak{P}(d)$ 
be the (derived) stack of representations of the preprojective algebra of $Q$ of dimension $d\in\mathbb{N}^I$. The category $\bigoplus_{d\in\mathbb{N}^I} D^b\left(\mathfrak{P}(d)\right)$
admits a Hall-type product \cite{VV}. The algebra obtained by taking $K_0$ of this category is called the  preprojective K-theoretic Hall algebra (KHA) of $Q$. It was studied by Schiffmann--Vasserot \cite{sv1}, Varagnolo--Vasserot \cite{VV}, Yang--Zhao \cite{yz2} and it was seen to be related to positive parts of quantum affine algebras. Conjecturally, an equivariant version of the preprojective KHA is isomorphic to the positive part of the Okounkov--Smirnov quantum affine algebra \cite{os}. It is thus natural to ask whether one can construct a full quantum group from a KHA.

For $Q$ a quiver, 
there is a tripled quiver with potential $\left(\widetilde{Q}, \widetilde{W}\right)$ whose Hall algebra is equivalent, via the Köszul equivalence/ dimensional reduction for categories of singularities \cite{I}, to the preprojective Hall algebra of $Q$.

\subsection{The coproduct}\label{coproduct}

For $d\in\mathbb{N}^I$, write $\X(d)=R(d)/G(d)$. 
For $\varepsilon\in E$, let $\mathbb{C}^*$ act on $R(d)$ by scalar multiplication of the linear map corresponding to $\varepsilon$.
Let $T\subset \left(\mathbb{C}^*\right)^E$ be a torus such that $W$ is $T$-invariant.  Denote by $q_\varepsilon$ the $T$-weight corresponding to $\varepsilon\in E$.

Let $\text{MF}_{\text{qcoh}, T}\left(\X(d)\times \X(e), W\right)$ be the category of quasi-coherent matrix factorizations \cite[Section 2.2]{T}, let $\lambda$ be a cocharacter of $G(d)\times G(e)$ which acts with non-negative weights on $R(d,e)$ and has fixed locus $R(d)\times R(e)$, and 
let \[\text{MF}_{\text{qcoh}, T}\left(\X(d)\times\X(e), W\right)_{\text{above}}\subset \text{MF}_{\text{qcoh}, T}\left(\X(d)\times\X(e), W\right)\] be the subcategory of quasi-coherent matrix factorizations
with $\lambda$-weights bounded above and with objects whose $\lambda$-weight $w$ components are in $\text{MF}_T(\X(d), W)_w$. 
Consider the functor:
\[\Delta'_{d,e}:=q_{d,e*}p_{d,e}^*: \text{MF}_T\left(\X(d+e), W\right)\to 
\text{MF}_{\text{qcoh}, T}\left(\X(d)\times\X(e), W\right)_{\text{above}}.\]
Let $T(d)$ be a maximal torus of $G(d)$ and let \[K_0^{T\times T(d)}(\text{pt})=K_0^T(\text{pt})\left[z_{ij}^{\pm 1}\Big|\,i\in I, 1\leq j\leq d_i\right].\]
Consider the set $\mathcal{I}_{d,e}$ of $K_0^{T\times T(d)\times T(e)}(\text{pt})$ with functions $1-q^{-1}_\varepsilon z_{i'j'}^{-1}z_{ij}$ for $i,i'\in I$,
$j'>d_{i'}$, $j\leq d_{i},$ and $q_\varepsilon$ is a weight of $T$ corresponding to an edge $\varepsilon$ from $i'$ to $i$.
The functor $\Delta'_{d,e}$ induces a map
\begin{equation}\label{cp}
\Delta'_{d,e}: K_0^T\left(\text{MF}(\X(d+e), W)\right)\to K_0^T\left(\text{MF}(\X(d)\times \X(e), W)\right)_{\mathcal{I}_{d,e}}.
\end{equation}
We make the following \textbf{Kunneth Assumption} on $(Q,W)$ and $T$:
\begin{equation}\label{assumption}
K_0^T\left(D_{\text{sg}}(\X(d)_0)\right)\otimes K_0^T\left(D_{\text{sg}}(\X(e)_0)\right)\cong K_0^{T\times T}\big(D_{\text{sg}}\left((\X(d)\times\X(e))_0\right)\big)
\end{equation}
for any dimension vectors $d$ and $e$;
see in this direction the Thom-Sebastiani theorem \cite{Pr}.
The pair $(Q,0)$ and any torus $T$ satisfy the Kunneth Assumption. The tripled quiver $\left(\widetilde{Q}, \widetilde{W}\right)$ and $T\subset (\mathbb{C}^*)^{E}\times\mathbb{C}^*_q$, see Subsection \ref{trqu}, satisfy the Kunneth Assumption after tensoring the two sides of \ref{assumption} by $\mathrm{Frac}\,K_0(BT)$. 

The functor $\Delta'_{d,e}$ (together with the equivalence $\text{MF}\cong D_{\text{sg}}$) induces a map into a completion, see Subsections \ref{comple} and \ref{cocompl} for more details:
\[\Delta'_{d,e}: K_0^T\left(D_{\text{sg}}(\X(d+e)_0)\right)\to K_0^T\left(D_{\text{sg}}(\X(d)_0)\right)\hat{\otimes} K_0^T\left(D_{\text{sg}}(\X(e)_0)\right).\]

\subsection{The extended algebra}
Let $\textbf{A}^{\geq}_T$ be the extended Hall algebra generated by $\text{KHA}_T(Q,W)$ and elements $h^+_{i,n}$ for $i\in I$, $n\geq 0$ with the relations \eqref{rel}. 
Denote by $\textbf{A}^{\geq}_T(d)$ the subspace generated by $K_0^T\left(D_{\text{sg}}(\X(d)_0)\right)$ and $h^+_{i,n}$. Consider the formal series $h_i(w):=\sum_{n\geq 0} h^+_{i,n}w^{-n}$. 
Define the coproduct:
\begin{align*}
\Delta_{d,e}: \textbf{A}^{\geq}_T(d+e)&\to \textbf{A}^{\geq}_T(d)\hat{\otimes} \textbf{A}^{\geq}_T(e)\\
x&\mapsto \left(\prod_{\substack{i\in I\\ j>d_i}} 
h_i(z_{ij})\right)* \Delta'_{d,e}(x).
\end{align*}


Consider also the algebra $\textbf{A}^{\leq}_T$ generated by  $\text{KHA}_T(Q,W)^{\text{op}}$ and generators $h^-_{i,n}$ for $i\in I$, $n\geq 0$, see Subsection \ref{exal}, which also admits a coproduct $\Delta$. 


\begin{thm}\label{thm1}
Let $(Q,W)$ be a symmetric quiver with potential. Let $T$ be a torus such that $W$ is $T$-invariant and such that $(Q,W)$ and $T$ satisfy the Kunneth Assumption \eqref{assumption}. Then
$\left(\textbf{A}^{\geq}_T, m, \Delta\right)$ and $\left(\textbf{A}^{\leq}_T, m, \Delta\right)$ are bialgebras.
\end{thm}

One can also state a result only involving the $\text{KHA}_T$, but then one needs a twist of the multiplication $m\boxtimes m$ in order to formulate the compatibility between $m$ and $\Delta$, see \cite{d} for the  analogous theorem for CoHA and for a definition of the twist in cohomology.

When $Q$ is the Jordan quiver, $W$ is zero, and $T=\mathbb{C}^*$, then $\textbf{A}^{\geq}_T\cong U_q(L\mathfrak{b})$, where $\mathfrak{b}$ is the Borel subalgebra of $\mathfrak{sl}_2$.
For the tripled quiver $\left(\widetilde{Q},\widetilde{W}\right)$
associated to the Jordan quiver and $T=\left(\mathbb{C}^*\right)^2$, $\textbf{A}_T^{\geq}$ has a subalgebra isomorphic to $U_{q, t}^{\geq}\left(\widehat{\widehat{\mathfrak{gl}_1}}\right)$ with the product and coproduct constructed in \cite{n1}.

\subsection{The Hopf pairing}
We assume now that we are in the case of a tripled quiver $\left( \widetilde{Q}, \widetilde{W}\right)$. The construction also works for other pairs such as $(Q,0)$, but not for all pairs $(Q,W)$ satisfying the Kunneth Assumption.

For each $e\in E(Q)$, let $\mathbb{C}^*$ act by multiplication with weight $1$ (by $q_e$) on the linear map corresponding to $e$ and by multiplication with weight $-1$ (by $q_e^{-1}$) on the linear map corresponding to $\overline{e}$. 
Let $\mathbb{C}^*_q$ act by $q$ on linear maps corresponding to edges in $\overline{Q}$ and by $q^{-2}$ on the loops $\omega_i$ for $i\in I$. Denote by $\mathbb{C}^*_t\subset \left(\mathbb{C}^*\right)^E$ the group which acts with weight $t$ on the edges on $Q$ and with weight $t^{-1}$ on the edges on $\overline{Q}$.
Consider a torus $T$ such that \begin{equation}\label{torusT}
    \mathbb{C}^*_t\times\mathbb{C}^*_q\subset T\subset (\mathbb{C}^*)^{E}\times\mathbb{C}^*_q.\end{equation}
These equivariant KHAs are called deformed KHAs in \cite{VV}. Then $\widetilde{R(d)}^T$ is a point. Let $\mathbb{F}$ be the fraction field of $K_0(BT)$. For a $K_0(BT)$-module $M$, we let $M_{\mathbb{F}}:=M\otimes_{\mathbb{K}}\mathbb{F}$. In Subsection \ref{pair}, we define a pairing
\[(\,,\,): K_0^T\left(D_{\text{sg}}\left(\widetilde{\X(d)}_0\right)\right)_{\mathbb{F}}\otimes_{\mathbb{F}} K_0^T\left(D_{\text{sg}}\left(\widetilde{\X(d)}_0\right)\right)_{\mathbb{F}}
\to 
\mathbb{F}.\] The definition is a direct generalization of the pairing used by Negu\c{t} in \cite[Exercise IV.2]{n2} for the cyclic type $A$ quiver.
The pairing $(\,,\,)$ is extended naturally to $\textbf{A}^{\leq}_{T,\mathbb{F}}\otimes_{\mathbb{F}} \textbf{A}^{\geq}_{T,\mathbb{F}}$.

\begin{thm}\label{thm2}
The pairing $(\,,\,)$ induces a non-degenerate pairing   
\[(\,,\,): \textbf{A}^{\leq}_{T,\mathbb{F}}\otimes_\mathbb{F} \textbf{A}^{\geq}_{T,\mathbb{F}}
\to 
\mathbb{F}.\]
Let $x,x'\in\textbf{A}^{\leq}_{T,\mathbb{F}}$, $y,y'\in\textbf{A}^{\geq}_{T,\mathbb{F}}$. 
Then
\begin{align*}
    (x*x',y)&=(x\otimes x', \Delta(y))\\
    (x, y*y')&=(\Delta^{\text{op}}(x), y\otimes y').
\end{align*}
\end{thm}

Negu\c{t} proved versions of Theorems \ref{thm1} and \ref{thm2}  for the preprojective KHA of the Jordan quiver in \cite{n1}, of the cyclic quiver of type $A$ \cite{n2}, for quivers with symmetric Cartan matrices in \cite{n4}, and for arbitrary quivers in \cite{n5}. 
An important ingredient in the proof of Theorem \ref{thm2} is the spherical generation of the preprojective KHA proved also by Negu\c{t} \cite{n3}.
Yang--Zhao \cite{yz1} proved versions of Theorems \ref{thm1} and \ref{thm2} for spherical subalgebra of preprojective Hall algebras.

Analogous to the constructions of Negu\c{t} in \cite[pages 88-89]{n2}, there are antipode maps $S:\textbf{A}_{T,\mathbb{F}}^{\geq}\to \textbf{A}_{T,\mathbb{F}}^{\geq}$, $S:\textbf{A}_{T,\mathbb{F}}^{\leq}\to \textbf{A}_{T,\mathbb{F}}^{\leq}$ which make $\textbf{A}_{T,\mathbb{F}}^{\geq}$ and $\textbf{A}_{T,\mathbb{F}}^{\leq}$ into Hopf algebras.
One can thus construct a Hopf algebra $\textbf{A}_{T,\mathbb{F}}$ as a 
Drinfeld double \cite[Subsection 1.4]{n2}.

It is interesting to see whether the representations constructed in \cite[Section 4]{P} can be extended to representations of $\textbf{A}_{T,\mathbb{F}}$. In the case of (the tripled quiver associated to) the the Jordan quiver, we expect some of those representations to be related to work of Schiffmann--Vaserot \cite{sv1} and Feigin--Tsymbaliuk \cite{ft} who constructed actions of $U_{q, t}\left(\widehat{\widehat{\mathfrak{gl}_1}}\right)$ on \[\bigoplus_{d\geq 0} K_0^T\left(\text{Hilb}(\mathbb{A}^2_{\mathbb{C}}, d)\right).\] It is also interesting to see whether there are natural categorifications of the double algebras $\textbf{A}_T$.

\subsection{Plan of the paper} In Section \ref{s2}, we discuss definitions and preliminary results. In Section \ref{s3}, we prove shuffle type formulas for $m$ and $\Delta$ and we use them to prove Theorem \ref{thm1}. We next show that the KHA for tripled quivers $\left(\widetilde{Q}, \widetilde{W}\right)$ satisfies the Kunneth Assumption. In Section \ref{s4} we define the Hopf pairing $(\,,\,)$ and prove Theorem \ref{thm2}. We end Section \ref{s4} with some explicit examples.

\subsection{Acknowledgements.} I thank the Institute of Advanced Studies for support during the preparation of the paper. I thank Henry Liu and Andrei Negu\c{t} for comments and corrections on a previous version of the paper.
This material is based upon work supported by the National Science Foundation under Grant No. DMS-1926686. I thank the referees for numerous useful suggestions that improved the paper.

\section{Preliminaries}\label{s2}

\subsection{Notations}
All stacks considered in the paper are defined over $\mathbb{C}$.

In the definition of the categorical preprojective Hall algebra, we need to use quasi-smooth schemes, a particular example of derived schemes, see \cite[Subsection 3.1.1]{T} for definitions and references on quasi-smooth schemes and stacks. 
For a stack $\X=X/G$ where $G$ is a reductive group acting on a (possibly quasi-smooth) quasi-projective scheme $X$, denote by $\text{QCoh}(\X)$ the unbounded derived category of quasi-coherent sheaves on $\X$, by $D^b\text{Coh}(\X)$ the derived category of bounded complexes of coherent sheaves, and by $\text{Perf}(\X)$ the subcategory of $D^b\text{Coh}(\X)$ of perfect complexes. All the functors used in the paper are derived and we drop $R$ and $L$ from their notations. 
Denote by 
\begin{align*}
    G_0(\X)&:=K_0\left(D^b\text{Coh}(\X)\right)_{\mathbb{Q}},\\
    K_0(\X)&:=K_0\big(\text{Perf}(\X)\big)_{\mathbb{Q}}.
\end{align*}
We denote by $H^{\text{BM}}_{\cdot}(\X)$ the Borel-Moore homology of $\X$ with rational coefficients. 
For $\X$ a quasi-smooth stack, denote by $\X^{\text{cl}}$ its classical stack. For the stacks considered in this paper, there is an isomorphism $G_0(\X)\cong G_0\left(\X^{\text{cl}}\right)$, see \cite[Equation 2.2]{VV}. We will be using the notations
\begin{align*}
    K_0^T\big(D_{\text{sg}}(\X(d)_0)\big)&:=K_0\big(D_{\text{sg},T}(\X(d)_0)\big)_{\mathbb{Q}}.\\
    K_0^T\left(\text{MF}(\X(d), W)\right)&:=K_0\left(\text{MF}_T(\X(d), W)\right)_{\mathbb{Q}}.
\end{align*}
For a regular immersion $\iota:\X\hookrightarrow \X'$, denote by $N_\iota$ the normal bundle of $\X$ in $\X'$.

All the potential appearing in the paper are induced from $\text{Tr}\,W$; we slightly abuse notation and write $W$ without mentioning the dimension vector. 
We use the notations $m$ and $*$ for the product of the KHA. 
All the tori $T$ which appear as equivariant parameters of the $\text{KHA}$ satisfy 
\begin{equation}\label{torus3}
T\subset \left(\mathbb{C}^*\right)^E\text{ and }W\text{ is }T\text{-invariant}.
\end{equation}
We also use in some situations an extra assumption \eqref{torusT} on tori $T$. 

\subsection{Categories of singularities.} 

A reference for this subsection is \cite[Subsection 2.2]{T}.
The construction of Hall algebras for quivers with potential involves categories of singularities \[D^T_{\text{sg}}(\X(d)_0):=D^b\text{Coh}_T\left(\X(d)_0\right)\big/\text{Perf}_T\left(\X(d)_0\right).\] Here, $\X(d)_0$ is the (derived) zero locus of $\text{Tr}\,W: \X(d)\to\mathbb{A}^1_\mathbb{C}$. We have an exact sequence
\begin{equation}\label{surjj}
    K_0^T(\X(d)_0)\to G_0^T(\X(d)_0)\to K_0^T\left(D_{\text{sg}}(\X(d)_0)\right)\to 0.
    \end{equation}

The category $D^T_{\text{sg}}(\X(d)_0)$ is equivalent to the category of matrix factorizations $\text{MF}_T\left(\X(d), W\right)$. The objects of $\text{MF}_T\left(\X(d), W\right)$ are $(\mathbb{Z}/2\mathbb{Z})\times T\times G(d)$-equivariant factorizations $(P, d_P)$, where $P$ is a $T\times G(d)$-equivariant coherent sheaf, $\langle 1\rangle$ is the twist corresponding to a non-trivial $\mathbb{Z}/2\mathbb{Z}$-character on $R(d)$, and \[d_P: P\to P\langle 1\rangle\] satisfies $d_P\circ d_P=\text{Tr}\,W$. There are natural pullback and proper pushforward functors for $\text{MF}$ induced by the natural pullback and proper pushforward functors on the ambient smooth stacks. We will freely switch between $D_{\text{sg}}$ and $\text{MF}$ throughout this paper.

Assume there is an extra $\mathbb{C}^*$-action corresponding to a subgroup of $\left(\mathbb{C}^*\right)^E$ such that $\text{Tr}\,W$ is of weight $2$. 
Consider the corresponding category of graded matrix factorizations $\text{MF}^{\text{gr}}(\X(d), W)$. 

For graded version of categories of singularities and for the quasi-coherent version of matrix factorizations, see \cite[Section 2.2]{T}; the only use of quasi-coherent matrix factorizations is to make sense of the functor $q_*p^*=\Delta'$ from Subsection \ref{coproduct}.

\subsection{Quivers}

\subsubsection{}
Let $Q=(I,E)$ be a quiver and let $d\in\mathbb{N}^I$. We denote by $\X(d)=R(d)/G(d)$ the stack of representations of $Q$ of dimension $d$. Fix maximal torus and Borel subgroups $T(d)\subset B(d)\subset G(d)$. We use the convention that the weights of the Lie algebra of $B(d)$ are negative; it determines a dominant chamber of weights of $G(d)$. For dimension vectors $a$ and $b$ with sum $d$, let $\lambda:\mathbb{C}^*\to G(d)$ be a fixed antidominant cocharacter corresponding to the partition $(a, b)$ of $d$. Consider the diagram of attracting loci for $\lambda$:
\[\X(a)\times\X(d)\cong\X(d)^\lambda\xleftarrow{q_{a,b}}\X(d)^{\lambda\geq 0}=\X(a,b)\xrightarrow{p_{a,b}}\X(d).\]
We may use the notations $p_\lambda, q_\lambda$ or $p, q$ depending on the context.

Let $W$ be a potential of $Q$ and let $T$ be a torus satisfying \eqref{torus3}. Then the $K$-theoretic Hall algebra of $(Q,W)$ and $T$ is
\[\text{KHA}_T(Q,W):=\bigoplus_{d\in\mathbb{N}^I}K_0^T\big(D_{\text{sg}}(\X(d)_0)\big).\]
The multiplication on Hall algebra \eqref{productHA} is induced by the functor
$p_{a,b*}q^*_{a,b}$ on the corresponding categories of singularities. 


\subsubsection{}\label{localtorus}
Let $d\in\mathbb{N}^I$ and let $T$ be a torus satisfying \eqref{torus3}. 
Consider the stacks $\mathcal{Y}(d):=R(d)/T(d)$. 
The map $a:\mathcal{Y}(d)\to \X(d)$, base change of $BT(d)\to BG(d)$, induces an injection 
\[a^*: K_0^T(D_{\text{sg}}(\mathcal{X}(d)_0))\hookrightarrow K_0^T\left(D_{\text{sg}}\left(\mathcal{Y}(d)_0\right)\right),\]
see \cite[Proposition 2.3]{P}.

\subsubsection{}\label{Phi}

Let $d\in\mathbb{N}^I$ and let $\lambda$ be a cocharacter of $G(d)$. Consider the natural maps
\[\mathcal{Y}(d)^\lambda:=R(d)^\lambda/T(d)\xrightarrow{s_\lambda} \mathcal{Y}(d)^{\lambda\geq 0}:=R(d)^{\lambda\geq 0}/T(d)\xrightarrow{p_\lambda} \mathcal{Y}(d)=R(d)/T(d).\]
We abuse notation and write $p_\lambda$ for the stacks $\mathcal{Y}$.
Define the map
\[\Phi_\lambda:=s_\lambda^*p_\lambda^*:K_0^T\left(D_{\text{sg}}(\mathcal{Y}(d)_0)\right)\to  K_0^T\left(D_{\text{sg}}\left(\mathcal{Y}(d)^\lambda_0\right)\right).\]
There are analogous maps denoted by the same letters
\[\X(d)^\lambda\xrightarrow{s_\lambda}\X(d)^{\lambda\geq 0}\xrightarrow{p_\lambda}\X(d).\]
The following diagram commutes:
\begin{equation}\label{diagram}
    \begin{tikzcd}
    K_0^T\left(D_{\text{sg}}(\X(d)_0)\right)\arrow[d, hook]\arrow[r,"\Phi_\lambda"]&
    K_0^T\left(D_{\text{sg}}(\X(d)^\lambda_0)\right)\arrow[d, hook]\\
    K_0^T\left(D_{\text{sg}}(\mathcal{Y}(d)_0)\right)\arrow[r,"\Phi_\lambda"]&K_0^T\left(D_{\text{sg}}(\mathcal{Y}(d)^\lambda_0)\right).
    \end{tikzcd}
\end{equation}

If $\lambda$ is the antidominant cocharacter corresponding to the partition $(a,b)$ of $d$, we may write $\Phi_{a,b}$ instead of $\Phi_\lambda$. Further, if $\lambda$ acts with weight $w$ on coordinates $z_{ij}$ parametrized by $A=(A_i)_{i\in I}$ for $A_i\subset \{1,\cdots, d_i\}$ and with a different weight $v\neq w$ on coordinates parametrized by its complement $B=(B_i)_{i\in I}$, we use the notation $\Phi_{A,B}$.

\subsection{The Köszul equivalence}\label{isikequi}
Let $X$ be a smooth affine scheme with an action of a reductive group $G$ and let $E$ be a $G$-equivariant vector bundle on $X$. Let $\mathbb{C}^*$ act on the fibers of $E$ with weight $2$ and consider $s\in \Gamma(X, E)$ a section of $E$ of weight $2$.
It induces a map $\partial: E^{\vee}\to \mathcal{O}_X$.
Define the $G$-equivariant regular function \[w:Y:=\text{Tot}_X\left(E^{\vee}\right)\to\mathbb{A}^1_\mathbb{C}\] by the formula
$w(x,v)=\langle s(x), v \rangle$ for $x\in X(\mathbb{C})$ and $v\in E^{\vee}|_x$.
Consider the Köszul stack
\[\mathfrak{K}:=\text{Spec}\left(\mathcal{O}_X\left[E^{\vee}[1];\partial\right]\right)\big/G.\]
Consider the category of graded matrix factorizations $\text{MF}^{\text{gr}}\left(Y/G, w\right)$ with respect to the group $\mathbb{C}^*$ mentioned above. The Köszul equivalence due to Isik \cite{I} is
\begin{equation}\label{koszu}
    \text{MF}^{\text{gr}}\left(Y/G, w\right)\cong D^b(\mathfrak{K}).
    \end{equation}
By \cite[Corollary 3.13]{T5} and \eqref{koszu}, we have that
\begin{equation}\label{grnotgr}
    K_0\left(\text{MF}(Y/G, w)\right)\cong
    K_0\left(\text{MF}^{\text{gr}}(Y/G, w)\right)\cong G_0\left(\mathfrak{K}\right).
\end{equation}

\subsection{The tripled quiver}
\label{trqu}
Let $Q=(I,E)$ be a quiver with vertices set $I$ and edges set $E$. Let $s,t:E\to I$ be the source and target maps. For $e\in E$, denote by $\overline{e}$ the edge with opposite orientation and let $\overline{E}=\{\overline{e}|\,e\in E\}$. For $i\in I$, let $\omega_i$ be a loop at $i$. Let $Q^d=(I,E^d)$ be the doubled quiver, where $E^d=E\sqcup \overline{E}$. Consider the tripled quiver $\widetilde{Q}=\left(I, \widetilde{E}\right)$, where $\widetilde{E}=\{f, \omega_i|\,f\in E^d, i\in I\}$, with potential $\widetilde{W}:=\sum_{e\in E}\omega_{s(e)}\,
[e, \overline{e}]$. 
\\


For $d\in\mathbb{N}^I$, let $\widetilde{\X(d)}$ be the stack of representations of $\widetilde{Q}$ of dimension $d$ and let $\mathfrak{P}(d)$ be the (derived) stack of representations of the preprojective algebra of $Q$ of dimension $d$ \cite{VV}, \cite[Subsection 3.2]{P}. By the Köszul equivalence
from Subsection \ref{isikequi}, see also loc. cit.: \begin{equation}\label{preequi}
\text{MF}^{\text{gr}}_{T}\left(\widetilde{\X(d)}, \widetilde{W}\right)\cong D^b_T\left(\mathfrak{P}(d)\right).\end{equation}
The preprojective categorical Hall algebra of $Q$ is the monoidal category
\[\text{HA}_T(Q):=\bigoplus_{d\in\mathbb{N}^I}D^b_T\left(\mathfrak{P}(d)\right),\]
see \cite{VV} for the definition of the product. Its Grothendieck group is called the preprojective KHA of $Q$.
The categories $\text{MF}^{\text{gr}}_{T}\left(\widetilde{\X(d)}, \widetilde{W}\right)$ and $\text{MF}_{T}\left(\widetilde{\X(d)}, \widetilde{W}\right)$ have the same Grothendieck group, see \cite[Corollary 3.13]{T5}, and thus \[\text{KHA}_T\left(\widetilde{Q}, \widetilde{W}\right)\cong\text{KHA}_T(Q)\] 
as vector spaces. The two algebras have the same product up to conjugation by an explicit equivariant element, see \cite[Subsection 3.2.3]{P}, \cite[Subsection 2.3.7]{VV}.

\subsection{The coproduct for KHA}\label{comple}
We explain why $\Delta_{d,e}'$ defined in Subsection \ref{coproduct} induces the map \eqref{cp}. In the setting of Subsection \ref{coproduct},
there is an orthogonal decomposition
\[
\text{MF}_{\text{qcoh}, T}\left(\X(d)\times\X(e), W\right)\cong\bigoplus_{w\in\mathbb{Z}}\text{MF}_{\text{qcoh}, T}\left(\X(d)\times\X(e), W\right)_w,
\]
where the subscript $w$ denotes the category of matrix factorizations on which $\lambda$ acts with weight $w$. The analogous decomposition holds for coherent matrix factorizations. 
There are thus functors \[\beta_w:\text{MF}_{\text{qcoh}, T}\left(\X(d)\times\X(e), W\right)_{\text{above}}\to \text{MF}_T\left(\X(d)\times\X(e), W\right)_{w}\] which associate to a matrix factorization $\mathcal{F}$ its $w$-weight component with respect to $\lambda$. 
Let $\mathcal{F}\in \text{MF}_T\left(\X(d+e), W\right)$. Then $q_*p^*(\mathcal{F})$ is in $\text{MF}_{\text{qcoh}, T}\left(\X(d)\times\X(e), W\right)_{\text{above}}$, see for example \cite[Lemma 2.2.3]{T}.
The formula for $\Delta'_{d,e}$ is:
\[\Delta'_{d,e}(\mathcal{F})=\sum_{w\in \mathbb{Z}} \left[\beta_w q_*p^*(\mathcal{F})\right]\in K_0\Big(\text{MF}_{\text{qcoh}, T}\left(\X(d)\times\X(e), W\right)_{\text{above}}\Big).\]
Then $\Delta'_{d,e}(\mathcal{F})$ is in $K_0\Big(\text{MF}_T\left(\X(d)\times\X(e), W\right)\Big)[\![q_\varepsilon^{-1}z_{i'j'}^{-1}z_{ij}]\!]$, where the monomials in the power series have $j\leq d_i$, $j'>d_{i'}$, and $q_\varepsilon$ is a weight of $T$ associated to an edge $\varepsilon$ from $i'$ to $i$.
Proposition \ref{p2}, part (b) says that this sum is the expansion of a rational function in $K_0\big(\text{MF}_T(\X(d)\times \X(e), W)\big)_{\mathcal{I}_{d,e}}$ in the region $|z_{ij}|\ll |z_{i'j'}|$ for $j\leq d_i$, $j'>d_{i'}$.

\section{The product-coproduct compatibility}\label{s3}

\subsection{Shuffle formulas for product and coproduct}\label{sh}

Let $i,i'$ be vertices and let $\{1,\cdots, \varepsilon(i,i')\}$ be the set of edges from $i$ to $i'$. Consider the action of $\left(\mathbb{C}^*\right)^{\varepsilon(i,i')}$ on $R(d)$ whose $j$th copy acts on $R(d)$ with weight $1$ on the factor $\text{Hom}\left(\mathbb{C}^{d_{s(j)}}, \mathbb{C}^{d_{t(j)}}\right)$ corresponding to the edge $j$. Denote by $q_j$ the weight corresponding to the $j$th copy of $\mathbb{C}^*$. 
Define
\begin{equation}\label{zeta}
    \zeta_{ii'}(z):= \frac{\left(1-q_1^{-1}z^{-1}\right)\cdots \left(1-q_{\varepsilon(i,i')}^{-1}z^{-1}\right)}{\left(1-z^{-1}\right)^{\delta_{ii'}}}, \end{equation}
where $\delta_{ii'}$ is $1$ if $i=i'$ and $0$ otherwise. 
Let $T$ be a subtorus of $\left(\mathbb{C}^*\right)^E$ which fixes $W$. We also use the notation $q_j$ for the corresponding weight of $T$. Recall the definition of $\Phi_{a,b}$ from Subsection \ref{Phi}.

\begin{prop}\label{p2}
Let $a,b\in\mathbb{N}^I$ with sum $d$. 

(a) Let $x\in K_0^T(D_{\text{sg}}(\X(a)_0))$ and $y\in K_0^T(D_{\text{sg}}(\X(b)_0))$. Then 
\[\Phi_{a,b}\,m_{a,b}(x,y)=\sum_{w\in\mathfrak{S}_{d}/\mathfrak{S}_a\times\mathfrak{S}_b} w\left(xy\prod_{\substack{i,i'\in I\\ j\leq a_i\\ j'>a_{i'}}}\zeta_{ii'}\left(\frac{z_{ij}}{z_{i'j'}}\right)
\right).\]
(b) Assume that $(Q,W)$ and $T$ satisfy the Kunneth Assumption \eqref{assumption}. Let $z$ be in $K_0^T\big(D_{\text{sg}}(\X(d)_0)\big)$. Then
\[\Delta'_{a,b}(z)=\frac{\Phi_{a,b}(z)}{\prod_{\substack{i,i'\in I\\ j>a_i\\ j'\leq a_{i'}}}
\zeta_{ii'}\left(\frac{z_{ij}}{z_{i'j'}}\right)}.\]
\end{prop}

\begin{proof} Consider the diagonal action of $T$ on $\X(a)\times \X(b)$. We are using the notations from Subsection \ref{Phi}.

(a) The map $q_{a,b}:\X(a,b)\to\X(a)\times\X(b)$ is an affine bundle map.  Its restriction to the zero locus is also an affine bundle map, so there is a pullback map $q^*_{a,b}$ in $K$ and $G$-theory and it is an isomorphism. 
Further, $q^*_{a,b}\left(\text{Tr}\,W_a+\text{Tr}\,W_b\right)=p^*_{a,b}\left(\text{Tr}\,W_{d}\right)$, so there are maps
\begin{align*}
    q^*_{a,b}&: G_0^T\Big(\left(\X(a)\times\X(b)\right)_0\Big)
    \cong G_0^T(\X(a,b)_0)\\
    q^*_{a,b}&: K_0^T\Big(D_{\text{sg}}\left((\X(a)\times\X(b)\right)_0 \Big)
    \cong K_0^T(D_{\text{sg}}(\X(a,b)_0)).
\end{align*}
Consider the factorizations of $p_{a,b}$:
\begin{align*}
&R(a,b)/G(a,b)\xrightarrow{\iota'_{a,b}} R(d)/G(a,b)\xrightarrow{\pi'_{a,b}} R(d)/G(d)\\
&R(a,b)_0/G(a,b)\xrightarrow{\iota_{a,b}} R(d)_0/G(a,b)\xrightarrow{\pi_{a,b}} R(d)_0/G(d).
\end{align*}
We have that $N_{\iota'_{a,b}}\big|_{\X(a,b)_0}=N_{\iota_{a,b}}$. The weights of the bundle $N_{\iota'_{a,b}}^{\vee}$ are of the form $q^{-1}_{\varepsilon}z_{ij}^{-1}z_{i'j'}$, where $j\leq a_i$, $j'>a_{i'}$, and $\varepsilon\in\{1,\cdots,\varepsilon(i,i')\}$ is an edge between $i$ and $i'$. We have that \begin{align*}
    \iota_{a,b}^*\iota_{a,b*}: \text{MF}_T(\X(a,b), W)&\to \text{MF}_T(\X(a,b), W)\\
    (P, d_P)&\mapsto \left(\iota_{a,b}^*\iota_{a,b*}P, \iota_{a,b}^*\iota_{a,b*}d_P\right).
    \end{align*}
    We have that $\mathcal{H}^{-i}\left(\iota_{a,b}^*\iota_{a,b*}P\right)\cong \Lambda^i\left(N_{\iota_{a,b}}^{\vee}\right)\otimes P$ for $i\in\mathbb{Z}$ and the corresponding factorization is $1\otimes d_P$, so 
\[\iota_{a,b}^*\iota_{a,b*}: K_0^T\big(\text{MF}(\X(a,b), W)\big)\to K_0^T\big(\text{MF}(\X(a,b), W)\big)\]
has the formula
\begin{equation}\label{1}
\iota_{a,b}^*\iota_{a,b*}(u)=u\prod_{\substack{i,i'\in I\\ j\leq a_i\\ j'>a_{i'}}}\left(1-q_1^{-1} z_{ij}^{-1}z_{i'j'}\right)\cdots \left(1-q_{\varepsilon(i,i')}^{-1} z_{ij}^{-1}z_{i'j'}\right)
\end{equation}
for $u$ in $K_0^T(D_{\text{sg}}(\X(a,b)_0))$.
The map $\pi_{a,b}$ is smooth, so there is a pullback in $G$-theory. 
Using Proposition \ref{refe}, the maps 
\begin{align*}
    \pi_{a,b}^*\pi_{a,b*}: G_0^T\left(R(d)_0/G(a,b)\right)&\to G_0^T\left(R(d)_0/G(a,b)\right)\\
    \pi_{a,b}^*\pi_{a,b*}: K_0^T\Big(D_{\text{sg}}(R(d)_0/G(a,b))\Big)&\to K_0^T\Big(D_{\text{sg}}(R(d)_0/G(a,b))\Big)
\end{align*}
have the following formulas, for example for $v$ in $K_0^T\Big(D_{\text{sg}}(R(d)_0/G(a,b))\Big)$:
\begin{equation}\label{2}
\pi_{a,b}^*\pi_{a,b*}(v)=\sum_{w\in\mathfrak{S}_{d}/\mathfrak{S}_a\times\mathfrak{S}_b} w\left(\frac{v}{\prod_{\substack{i\in I\\ j\leq a_i<k}}\left(1-z_{ij}^{-1}z_{ik}\right)}\right).
\end{equation}
Note that Proposition \ref{refe} has a formula for $G_0^T\left(R(d)_0/G(a,b)\right)$, but the formula for $K_0^T\Big(D_{\text{sg}}(R(d)_0/G(a,b))\Big)$ follows using the surjection \eqref{surjj}.
Further, $\iota_{a,b}^*$ and the action of $\mathfrak{S}_d$ commute. 
The formula for $m_{a,b}$ follows from \eqref{1} and \eqref{2}.

(b) Consider the factorization of $q_{a,b}$: 
\[R(a,b)_0/G(a,b)\xrightarrow{r_{a,b}}\left(R(a)\times R(b)\right)_0/G(a, b)\xrightarrow{j_{a,b}}\left(R(a)\times R(b)\right)_0/G(a)\times G(b).\]
Let $\omega_{a,b}: \left(R(a)\times R(b)\right)_0/G(a)\times G(b)\to \left(R(a)\times R(b)\right)_0/G(a,b)$ be the map induced by the inclusion $G(a)\times G(b)\subset G(a,b)$.
The map $r_{a,b}$ has a natural section $\tau_{a,b}$.
There are induced maps
\begin{align*}
    j_{a,b*}: G_0^T(\X(a,b)_0)&\to G_0^T\left(R(a,b)_0/G(a)\times G(b)\right)\\
    j_{a,b*}: K_0^T(D_{\text{sg}}(\X(a,b)_0))&\to K_0^T\Big(D_{\text{sg}}(R(a,b)_0/G(a)\times G(b))\Big).
\end{align*}
The map $\omega_{a,b}$ is smooth, so there is a pullback map in $G$-theory.
The map $j_{a,b*}$ has the following formula, for example for $u$ in $G_0^T(\X(a,b)_0)$:
\begin{equation}\label{3}
j_{a,b*}(u)=\omega_{a,b}^*(u)\prod_{\substack{i\in I\\ k\leq a_i<j}}\left(1-z_{ij}^{-1}z_{ik}\right).
\end{equation}
Finally, $r_{a,b}$ is the restriction to the zero locus of the affine bundle map \[r'_{a,b}:R(a,b)\big/G(a)\times G(b)\to R(a)\times R(b)\big/G(a)\times G(b).\] 
The pullback $r^*_{a,b}$ exists and is an isomorphism in $K$ and $G$-theory.
Let $\lambda$ be a cocharacter of $G(a)\times G(b)$ which acts with non-negative weights on $R(a,b)$ and has fixed locus $R(a)\times R(b)$. Similarly to the discussion in Subsection \ref{comple},   $r_{a,b*}$ has image in
\begin{multline*}
r_{a,b*}:
K_0^T\Big(D_{\text{sg}}(R(a,b)_0/G(a)\times G(b))\Big)
\to\\
K_0^T\Big(D_{\text{sg}}(\left(R(a)\times R(b)\right)_0/G(a)\times G(b))\Big)[\![q_\varepsilon^{-1}z_{i'j'}^{-1}z_{ij}]\!],
\end{multline*} where the monomials in the power series have $j\leq a_i$, $j'>a_{i'}$.
Let
$N:=R(a,b)^{\lambda>0}$ be a representation of $G(a)\times G(b)$.
For a pair $(\mathcal{F},d)$ in $\text{MF}_T\big(\X(a)\times \X(b), W\big)$, we have that \[r_{a,b*}r_{a,b}^*\left(\mathcal{F}\right)=\mathcal{F}\otimes \mathbb{C}[N^{\vee}]\] and $r_{a,b*}r_{a,b}^*(d)=d\otimes \text{id}$. 
The map $\tau_{a,b}^*$ is an inverse to $r_{a,b}^*$. 
Consider $v$ in $K_0^T\big(\text{MF}\big(R(a,b)\big/G(a)\times G(b), W\big)\big)$. Then
\begin{equation}\label{4}
r_{a,b*}(v)=\frac{\tau_{a,b}^*(v)}{\prod_{\substack{i,i'\in I\\ j>a_i\\ j'\leq a_{i'}}}\left(1-q_1^{-1}z_{ij}^{-1}z_{i'j'}\right)\cdots\left(1-q_{\varepsilon(i,i')}^{-1}z_{ij}^{-1}z_{i'j'}\right)}.
\end{equation}
Consider the action of $T\times T$ on $\X(a)\times \X(b)$ induced by the actions of $T$ on both factors. By the Kunneth Assumption \eqref{assumption},
we have that
\[
K_0^{T\times T}\Big(D_{\text{sg}}((\X(a)\times\X(b))_0)\Big)\cong K_0^T(D_{\text{sg}}(\X(a)_0))\otimes
K_0^T(D_{\text{sg}}(\X(b)_0)).\]
The map $r_{a,b*}$ has image in 
$K_0^T\left(D_{\text{sg}}(\X(a)_0)\right)\hat{\otimes} K_0^T\left(D_{\text{sg}}(\X(b)_0)\right),$
where the completion on the right hand side is defined in Subsection \ref{comple}, see also Subsection \ref{cocompl}. 
The formula for $\Delta'_{a,b}$ follows by combining \eqref{3}, \eqref{4}, and $\Phi_{a,b}=p^*_{a,b}\omega^*_{a,b}\tau^*_{a,b}$.
\end{proof}

To conclude \eqref{2}, it suffices to show:

\begin{prop}\label{refe}
Let $a, b, d\in \mathbb{N}^I$ be such that $a+b=d$. Let $P:=G(a,b)$ be the parabolic group of $G:=G(d)$ corresponding to the partition $a+b=d$. Let $X$ be any variety with an action of $G$ and consider the smooth projection map \[\pi: X/P\cong X\times_P G/G\to X/G.\] Let $y\in G_0(X/P)$. Then \[\pi^*\pi_{*}(y)=\sum_{w\in\mathfrak{S}_{d}/\mathfrak{S}_a\times\mathfrak{S}_b} w\left(\frac{y}{\prod_{\substack{i\in I\\ j\leq a_i<k}}\left(1-z_{ij}^{-1}z_{ik}\right)}\right).\]
\end{prop}

\begin{proof}
First, the pullback $\pi^*$ exists because $\pi$ is smooth. The map $\pi$ is a relative Grassmannian.
We prove the statement using induction on the dimension of $X$. The statement for $X=\text{Spec}(\mathbb{C})$ is known, for example it follows from the Weyl character formula, see also \cite[Proposition 1.2]{yz2}. 

Assume that $X$ has dimension $n>0$. We may assume that $X$ is irreducible. 
By the relative version \cite[Subsection 3.7]{K2} of Kapranov's exceptional collection for the Grassmannian \cite{K1}, there exists a semi-orthogonal decomposition of $D^b(X\times_P G/G)$ with summands $\left(\pi^* D^b(X/G)\right)\otimes \mathcal{V}_i$ for $i$ in a finite set $B$, where $\mathcal{V}_i$ are equivariant vector bundles. It suffices to check the statement for sheaves $\left(\pi^*\mathcal{E}\right)\otimes \mathcal{V}$, where $\mathcal{V}$ is one such vector bundle. Let $\iota: Y\hookrightarrow X$ be the support of $\mathcal{E}$, meaning that $Y$ is the largest closed subvariety of $X$ such that $[\mathcal{E}]$ is in the image of $\iota_*: G_0(Y/G)\to G_0(X/G)$. 

Assume first that $Y$ is not $X$, so $\dim(Y)<\dim(X)$. Then $\left[\left(\pi^*\mathcal{E}\right)\otimes \mathcal{V}\right]$ is in the image of $\iota_*: G_0(Y/P)\to G_0(X/P)$, so the statement follows by the induction hypothesis. 

Assume next that the support of $\mathcal{E}$ is $X$. Then there exist a nonzero open $U\subset X$ and a representation $\Gamma$ of $G$ such that $\mathcal{E}|_U\cong \mathcal{O}_{U/G}\otimes \Gamma$. The statement for vector bundles $\left(\mathcal{O}_{X/G}\otimes \Gamma\right)\otimes \mathcal{V}$ follows from the case of $X=\text{Spec}(\mathbb{C})$, see also \cite[Proposition 1.2]{yz2}. It thus suffices to check the statement for the class $\left[\left(\pi^*\mathcal{E}\right)\otimes \mathcal{V}\right]-
\left[\left(\pi^*\left(\mathcal{O}_{X/G}\otimes\Gamma\right)\right)\otimes \mathcal{V}\right]\in G_0(X/G)$. The support of the class $[\mathcal{E}]-[\mathcal{O}_{X/G}\otimes\Gamma]$ is contained in $X\setminus U$, so it is not $X$, and thus the statement follows by the induction hypothesis. 
\end{proof}


\subsection{The extensions of the Hall algebra}
\label{exal}
Let $\textbf{A}^{\geq}_T$ be the extended Hall algebra generated by $\text{KHA}_T(Q,W)$ and generators $h^+_{i,n}$ for $i\in I$, $n\geq 0$.
Consider the formal power series $h^+_i(w):=\sum_{n\geq 0} h_{i,n}^+w^{-n}$. Define
\[\tau_{ii'}(z):=\frac{\zeta_{i'i}(z)}{\zeta_{ii'}(z^{-1})}.\]

We impose the relations 
\begin{equation}\label{rel}
    y*h^+_{i}(w)=h^+_{i}(w)*\left(y\prod_{\substack{i'\in I\\ 1\leq j'\leq d_{i'}}}\tau_{ii'}\left(\frac{z_{i'j'}}{w}\right)\right)
\end{equation}
for all $y$ in $K_0^T(D_{\text{sg}}(\X(d)_0))$ and $i\in I$. We also impose that $\left[h^+_{i}(w), h^+_{i'}(w')\right]=0$ for all $i,i'\in I$.
For references to explicit relations between $h_{i,n}^+$ and $\text{KHA}_T(Q,W)$, see Subsections \ref{ss1} and \ref{ss2}.

Similarly, define $\textbf{A}^{\leq}_T$ as the extended Hall algebra generated by $\text{KHA}_T(Q,W)^{\text{op}}$ and generators $h^-_{i,n}$ for $i\in I$, $n\geq 0$. Let $h^-_i(w):=\sum_{n\geq 0} h^-_{i,n}w^{n}$. Impose the relations 
\begin{equation}\label{rel2}
    y*h^-_{i}(w)=h^-_{i}(w)*\left(y\prod_{\substack{i'\in I\\ 1\leq j'\leq d_{i'}}}\tau_{ii'}\left(\frac{w}{z_{i'j'}}\right)\right)
\end{equation}
for all $y$ in $K_0^T(D_{\text{sg}}(\X(d)_0))$ and $i\in I$. We also impose that $\left[h^-_{i}(w), h^-_{i'}(w')\right]=0$ for all $i,i'\in I$. 

\subsection{The coproduct}

Let $a, b\in\mathbb{N}^I$ have sum $d$. Let $y$ be in $K_0^T(D_{\text{sg}}(\X(d)_0))$.

\subsubsection{} 
Define the coproduct for $\textbf{A}^{\geq}_T$:
\begin{align*}
\Delta\left(h^+_i(w)\right)&=h^+_i(w)\otimes h^+_i(w)\\
    \Delta_{a,b}(y)&=\prod_{\substack{i\in I\\ j>a_i}}h^+_i(z_{ij})*\Delta'_{a,b}(y)\\
    \Delta(y)&=\sum_{\substack{a,b\in\mathbb{N}^I\\a+b=d}}\Delta_{a,b}(y).
\end{align*}

\subsubsection{}
We next define the coproduct for $\textbf{A}^{\leq}_T$. 
Define 
\[\Delta'_{a,b}(y)=\text{sw}_{b,a} q_{b,a*}p_{b,a}^*(y),\]
where $\text{sw}_{b,a}$ is the map that interchanges the factors:
\[\text{sw}_{b,a}: K_0^T\left(D_{\text{sg}}(\X(b)_0)\right)\otimes K_0^T\left(D_{\text{sg}}(\X(a)_0)\right)\cong K_0^T\left(D_{\text{sg}}(\X(a)_0)\right)\otimes K_0^T\left(D_{\text{sg}}(\X(b)_0)\right).\]
$y$ is $\mathfrak{S}_{d}$-symmetric. A similar computation to Proposition \ref{p2}, part (b) shows that:
\begin{equation}\label{coprneg}
\Delta'_{a,b}(y)=\frac{\Phi_{a,b}(y)}{\prod_{\substack{i,i'\in I\\ j\leq a_i\\ j'>a_{i'}}}
\zeta_{ii'}\left(\frac{z_{ij}}{z_{i'j'}}\right)}.\end{equation}
We define the coproduct for $\textbf{A}^{\leq}_T$:
\begin{align*}
 \Delta\left(h^-_i(w)\right)&=h^-_i(w)\otimes h^-_i(w)\\
    \Delta_{a,b}(y)&=\Delta'_{a,b}(y)*\prod_{\substack{i\in I\\ j\leq a_i}}h^-_i(z_{ij})\\
    \Delta(y)&=\sum_{a,b\in\mathbb{N}^I}\Delta_{a,b}(y).
\end{align*}

\subsubsection{}\label{cocompl}
To make sense of the coproduct formulas, expand in the region $|z_{ij}|\ll |z_{i'j'}|$ for $j\leq a_i$, $j'>a_{i'}$, place all functions in $z_{ij}$ for $j\leq a_i$ to the left of $\otimes$ and all functions in $z_{ij}$ for $j>a_i$ to the right of $\otimes$. The coproduct has infinitely many terms, but each homogeneous $d$ part is in $\textbf{A}^{\geq}_T(d)$ or $\textbf{A}^{\leq}_T(d)$. We consider a completion of $\textbf{A}^{\geq }_T(d)\hat{\otimes} \textbf{A}^{\geq }_T(e)$ in which the coproduct makes sense. For example, we can consider the completion $\varprojlim_u K_0^T\left(\text{MF}(\X(d)\times \X(e), W)\right)/F_u$
of
$K_0^T\left(\text{MF}(\X(d)\times \X(e), W)\right)$. Here, $F_u$ is defined by \[F_u:=K_0^T\left(\text{MF}(\X(d)\times \X(e), W)_{\leq u}\right)\subset K_0^T\left(\text{MF}(\X(d)\times \X(e), W)\right)\] for $u\geq 1$, 
where $\text{MF}^T(\X(d)\times \X(e), W)_{\leq u}$ is the subcategory of $\text{MF}^T(\X(d)\times \X(e), W)$ of complexes on which $\lambda_{d,e}$ acts with weights $\leq u$.

\subsection{The counit maps}
The counit maps $\epsilon: \textbf{A}^{\geq}\to \textbf{k}$, $\epsilon: \textbf{A}^{\leq}\to \textbf{k}$ are defined by $\epsilon\left(y*h^{\pm}_i(z)\right)=0$ for $y$ nonzero in $\text{KHA}_T(Q,W)$ and $i\in I$ and by $\epsilon\left(h^{\pm}_i(z)\right)=1$. 

\subsection{Proof of Theorem \ref{thm1}}
$\text{KHA}_T$ is an algebra by \cite[Theorem 3.3]{P}. $\textbf{A}^{\geq}_T$ and $\textbf{A}^{\leq}_T$ are also algebras with unit $1$ in degree $0$. By a direct computation using Proposition \ref{p2} or by an argument similar to the one in \cite[Theorem 3.3]{P}, $\text{KHA}_T$, $\textbf{A}^{\geq}_T$, and $\textbf{A}^{\leq}_T$ are also coalgebras with counit $\epsilon$. We drop the suprascript from the notation of the formal series $h_i$, $i\in I$.

Let $c$ and $f$ be in $\mathbb{N}^I$. Let $\textbf{S}$ be the set of quadruples $(e_1, e_2, e_3, e_4)$ such that 
\[e_1+e_2=a,\,\, e_3+e_4=b,\,\, e_1+e_3=c,\,\, e_2+e_4=f.\]
We need to show that 
\begin{equation*}
    \Delta_{c,f}(x*y)=\sum_{\textbf{S}}\Delta_{e_1,e_2}(x)*\Delta_{e_3,e_4}(y).
\end{equation*}
Let $d\in\mathbb{N}^I$. Let $\lambda$ be a generic cocharacter of $G(d)$. Then \[\X(d)^\lambda\cong \mathcal{Y}(d)^\lambda= \times_{i\in I} \mathcal{Y}(\delta_i)^{\times d_i}\cong \times_{i\in I} \X(\delta_i)^{\times d_i} ,\] where $\delta_i$ is the dimension vector with $1$ in position $i$ and zero otherwise. Consider the inclusion
\[\iota_d:\mathcal{Y}(d)^\lambda\hookrightarrow \mathcal{Y}(d).\]
Let $\mathcal{I}$ be the set in $K_0^{T\times T(d)}(\text{pt})$ of functions $1-q^{-1}_\varepsilon z_{ij}^{-1}z_{i'j'}$ for $i,i'\in I$,
$j\leq d_i$, $j'\leq d_{i'}$ with $(ij)\neq (i'j')$, and $q_\varepsilon$ is a weight of $T$ corresponding to an edge from $i$ to $i'$. Then, by the localization theorem \cite[Theorem 2.5]{P}, 
\[\iota^*_d: K_0^T\big(D_{\text{sg}}(\mathcal{Y}(d)_0)\big)_\mathcal{I}\cong K_0^T\left(D_{\text{sg}}\left(\mathcal{Y}(d)^\lambda_0\right)\right)_\mathcal{I}.\] 
Let $\mathcal{N}$ be the normal bundle of $\mathcal{Y}(d)^\lambda$ in $\mathcal{Y}(d)$. There is a natural projection map \[u:\mathcal{Y}(d)\cong\text{Tot}\left(\mathcal{N}\right)\to \mathcal{Y}(d)^\lambda.\]
Consider the action of $\mathbb{C}^*$ on $\text{Tot}\left(u^*\mathcal{N}\right)$ which is the pullback of the action of $\mathbb{C}^*$ via $\lambda$ on $\text{Tot}\left(\mathcal{N}\right)$. We have that $\text{Tot}\left(u^*\mathcal{N}\right)^\lambda=\mathcal{Y}(d)$.
By \cite[Proposition 2.4]{P} for the vector bundle $u^*\mathcal{N}$ on $\mathcal{Y}(d)$ and the above action of $\mathbb{C}^*$, the following map is injective \[K_0^T\left(D_{\text{sg}}(\mathcal{Y}(d)_0)\right)\hookrightarrow K_0^T\left(D_{\text{sg}}(\mathcal{Y}(d)_0)\right)_\mathcal{I}.\] 
We write $\iota_d$ for maps from $\mathcal{Y}$ and $\X$; we write $\iota_A$ for maps from stacks $\mathcal{Y}_A$ with coordinates in the set $A=(A_i)_{i\in I}$.
It thus suffices to show that
\begin{equation}\label{prcopr}
    \left(\iota^*_c\otimes \iota_f^*\right)\Delta_{c,f}(x*y)=\left(\iota^*_c\otimes \iota_f^*\right)\sum_{\textbf{S}}\Delta_{e_1,e_2}(x)*\Delta_{e_3,e_4}(y).
\end{equation}

Let $A=(A_i)_{i\in I}$ and $B=(B_i)_{i\in I}$ be sets such that $A_i\sqcup B_i=\{1,\cdots, a_i+b_i\}$, $|A_i|=a_i$, $|B_i|=b_i$. By Proposition \ref{p2}, part (a), we have that 
\[\Phi_{a,b}(x*y)=\sum_{A,B}\left(x_Ay_B\prod_{\substack{i,i'\in I\\ j\in A_i\\ j'\in B_{i'}}}\zeta_{ii'}\left(\frac{z_{ij}}{z_{i'j'}}\right)\right),\]
where the sum is after all sets $A$ and $B$ as above.
Define the $I$-tuples of sets $A^1=\left(A^1_i\right)_{i\in I}$, $A^2=\left(A^2_i\right)_{i\in I}$, $B^1=\left(B^1_i\right)_{i\in I}$, $B^2=\left(B^2_i\right)_{i\in I}$:
\begin{align*}
  A^1_i&=A_i\cap \{1,\cdots, c_i\},\\
  A^2_i&=A_i\cap \{c_i+1,\cdots, c_i+f_i\},\\
  B^1_i&=B_i\cap \{1,\cdots, c_i\},\\
  B^2_i&=B_i\cap \{c_i+1,\cdots, c_i+f_i\}.
\end{align*}
We introduce the following sets of indices:
\begin{align*}
  \textbf{E}&=\{(i,i',j,j')|\,i,i'\in I,  (j,j')\in (A^1_i,B^1_{i'}), (A^1_i,B^2_{i'}), (A^2_i,B^2_{i'})\}\\
  \textbf{F}&=\{(i,i',j,j')|\,i,i'\in I,  (j,j')\in (A^2_i,A^1_{i'}), (B^2_i,A^1_{i'}), (B^2_i,B^1_{i'})\}\\
  \textbf{G}&=\{(i,i',j,j')|\,i,i'\in I,  (j,j')\in (A^1_i,B^1_{i'}),  (A^2_i,B^2_{i'})\}\\
  \textbf{H}&=\{(i,i',j,j')|\, i,i'\in I, (j,j')\in (A^1_i, B^2_{i'})\}\\
  \textbf{K}&=\{(i,i',j,j')|\, i,i'\in I, (j, j')\in (B^2_i, A^1_{i'})\}.
\end{align*}
Write $\Phi_{A^1,A^2}(x)=\sum x_{A^1}\otimes x_{A^2}$; we will drop the sum sign in the formula of $\Phi$ in what follows. 
We have that: 
\begin{align*}
    \left(\iota^*_c\otimes \iota_f^*\right)
    \Delta_{c,f}(x*y)
    &=\prod_{\substack{i\in I\\j>c_i}}
    h_i(z_{ij})* 
    \frac{\sum_{A,B}
    \iota^*_c(x_{A^1}y_{B^1})
    \otimes
    \iota_f^*(x_{A^2}y_{B^2})
    \prod_{\substack{i,i'\in I\\ j\in A_i\\ j'\in B_{i'}}}\zeta_{ii'}\left(\frac{z_{ij}}{z_{i'j'}}\right) 
    }
    {\prod_{\substack{i,i'\in I\\j>c_i\\j'\leq c_{i'}}}\zeta_{ii'}\left(\frac{z_{ij}}{z_{i'j'}}\right)}\\
    &=\prod_{\substack{i\in I\\j\in A^2_i\sqcup B^2_i}}h_i(z_{ij})*\sum_{\substack{j\in A^1_i\sqcup A^2_i\\j'\in B^1_{i'}\sqcup B^2_{i'}}} 
    \iota^*_c(x_{A^1}y_{B^1})\otimes \iota_f^*(x_{A^2}y_{B^2})
    \frac{\prod_{\textbf{E}} \zeta_{ii'}\left(\frac{z_{ij}}{z_{i'j'}}\right)}{\prod_{\textbf{F}} \zeta_{ii'}\left(\frac{z_{ij}}{z_{i'j'}}\right)}.
\end{align*}
Using Proposition \ref{p2}, part (b), we have that:
\begin{multline*}
    \Delta_{e_1,e_2}(x)*\Delta_{e_3,e_4}(y)=
    \left(
    \frac{ \prod_{\substack{i\in I\\j>e_{1i}}}h_i(z_{ij})*\Phi_{e_1,e_2}(x)}{
    \prod_{\substack{i,i'\in I\\j>e_{1i}\\j'\leq e_{1i'}}}\zeta_{ii'}\left(\frac{z_{ij}}{z_{i'j'}}\right)
    }
    \right)*
    \left(
    \frac{\prod_{\substack{i\in I\\j>a_i+e_{3i}}}h_i(z_{ij})*\Phi_{e_3,e_4}(y)}{
    \prod_{\substack{i,i'\in I\\j>a_i+e_{3i}\\j'\leq a_{i'}+e_{3i'}}}\zeta_{ii'}\left(\frac{z_{ij}}{z_{i'j'}}\right)
    }
    \right)
    \end{multline*}
Using that $\iota^*$ and $\Phi$ are compatible and using the formula for the (restriction of the) product from Proposition \ref{p2}, part (a), we obtain: 
\begin{multline*}
    \left(\iota^*_c\otimes \iota_f^*\right)\big(\Delta_{e_1,e_2}(x)*\Delta_{e_3,e_4}(y)\big)=
    \sum_{A^1,\cdots, B^2}
    \left(
    \frac{\prod_{\substack{i\in I\\j\in A^2_i}}h_i(z_{ij})*\left(\iota^*_{A^1}\otimes \iota^*_{A^2}\right)\left(\Phi_{A^1,A^2}(x)\right)}{
    \prod_{\substack{i,i'\in I\\j\in A^2_i\\j'\in A^1_{i'}}}\zeta_{ii'}\left(\frac{z_{ij}}{z_{i'j'}}\right)
    }
    \right)\\
    \left(
    \frac{\prod_{\substack{i\in I\\j\in B^2}}h_i(z_{ij})*\left(\iota^*_{B^1}\otimes \iota^*_{B^2}\right)\left(\Phi_{B^1,B^2}(y)\right)
    }{
    \prod_{\substack{i,i'\in I\\j\in B^2_i\\j'\in B^1_{i'}}}\zeta_{ii'}\left(\frac{z_{ij}}{z_{i'j'}}\right)
    }
    \right)
    \prod_{\textbf{G}}\zeta_{ii'}\left(\frac{z_{ij}}{z_{i'j'}}\right),
\end{multline*}
where the sum after all sets $A^1,\cdots, B^2$ as above.
We commute $x_{A^1}$ past the $h_i$s using the relations \eqref{rel}:
\begin{equation}\label{a}
x_{A^1}*\prod_{\substack{i'\in I\\j'\in B^2_{i'}}} h_{i'}(z_{i'j'})=
\prod_{\substack{i'\in I\\j'\in B^2_{i'}}} h_{i'}(z_{i'j'})*\left(x_{A^1}\prod_{\textbf{H}}
\tau_{i'i}\left(\frac{z_{ij}}{z_{i'j'}}\right)\right).
\end{equation}
We have that:
\begin{equation}\label{b}
\prod_{\textbf{H}}
\tau_{i'i}\left(\frac{z_{ij}}{z_{i'j'}}\right)=
\prod_{\textbf{H}}\frac{\zeta_{ii'}\left(\frac{z_{ij}}{z_{i'j'}}\right)}{\zeta_{i'i}\left(\frac{z_{i'j'}}{z_{ij}}\right)}=
\frac{
\prod_{\textbf{H}} \zeta_{ii'}\left(\frac{z_{ij}}{z_{i'j'}}\right)}
{\prod_{\textbf{K}} \zeta_{ii'}\left(\frac{z_{ij}}{z_{i'j'}}\right)}.
\end{equation}
Combining the formula in \eqref{b} with the formula for $\Delta_{e_1,e_2}(x)*\Delta_{e_3,e_4}(y)$, we obtain that:
\begin{multline*}
    (\iota^*_c\otimes\iota^*_f)\left(\Delta_{e_1,e_2}(x)*\Delta_{e_3,e_4}(y)\right)=
    \sum_{A^1,\cdots, B^2}\left(
    \prod_{\substack{i\in I\\j\in A^2_i}}h_i(z_{ij})
    \prod_{\substack{i\in I\\j\in  B^2_i}}h_i(z_{ij})\right)*\\
    \sum_{\substack{j\in A^1_i\sqcup A^2_i\\j'\in B^1_{i'}\sqcup B^2_{i'}}} \iota^*_c\left(x_{A^1}y_{B^1}\right)\otimes \iota_f^*\left(x_{A^2}y_{B^2}\right)
    \frac{\prod_{\textbf{E}} \zeta_{ii'}\left(\frac{z_{ij}}{z_{i'j'}}\right)}{\prod_{\textbf{F}} \zeta_{ii'}\left(\frac{z_{ij}}{z_{i'j'}}\right)}.
\end{multline*}
We thus obtain the equality claimed in \eqref{prcopr}.

\subsection{The Kunneth assumption for the tripled quiver}
For $Q$ a quiver, recall the definition of the tripled quiver from Subsection \ref{trqu}. Consider a torus $T\subset (\mathbb{C}^*)^{E}\times\mathbb{C}^*_q$, see Subsection \ref{trqu}. Let $\mathbb{K}:=K_0(BT)$ and let $\mathbb{F}:=\mathrm{Frac}\,\mathbb{K}$. For $M$ a $\mathbb{K}$-module, let $M_{\mathbb{F}}:=M\otimes_{\mathbb{K}}\mathbb{F}$.
We following is immediate using the localization theorem in K-theory and the Koszul equivalence, see Subsection \ref{trqu}:

\begin{prop}\label{propo}
The tripled quiver $\left(\widetilde{Q}, \widetilde{W}\right)$ and the torus $T$ satisfy the $\mathbb{F}$-localized Kunneth Assumption:
\[K_0^T\left(D_{\text{sg}}(\X(d)_0)\right)_{\mathbb{F}}\otimes_{\mathbb{F}} K_0^T\left(D_{\text{sg}}(\X(e)_0)\right)_{\mathbb{F}}\cong K_0^{T\times T}\big(D_{\text{sg}}\left((\X(d)\times\X(e))_0\right)\big)_{\mathbb{F}}.\]
\end{prop}

It is an interesting problem to determine whether the Kunneth property holds integrally for the tripled quiver $\left(\widetilde{Q}, \widetilde{W}\right)$ and the torus $T$.

\section{The Hopf pairing}\label{s4}

\subsection{The antipode map} 
Let $(Q,W)$ be an arbitrary symmetric quiver with potential and let $T$ be a torus satisfying \eqref{torus3}. We add inverses $\left(h^{\pm}_{i,0}\right)^{-1}$ to $h^{\pm}_{i,0}$ for $i\in I$ in the algebras $\textbf{A}^{\leq}_T$ and $\textbf{A}^{\geq}_T$. 
Then the relations \eqref{rel} and \eqref{rel2} induce natural relations involving $\left(h^{\pm}_{i,0}\right)^{-1}$. The counit map is 
$\epsilon\left(\left(h^{\pm}_{i,0}\right)^{-1}\right)=1$.
The series $h^{\pm}_{i}(z)$ are invertible because $h^{\pm}_{i,0}$ are invertible.
Let $y$ be in $K_0^T(D_{\text{sg}}(\X(d)_0))$. The antipode map $S$ for $\textbf{A}^{\geq}_T$ is defined by:
\begin{align*}
    S\left(h^+_i(z)\right)&=\left(h^+_i(z)\right)^{-1}\\
    S(y)&=\left[\prod_{\substack{i\in I\\j\leq d_i}}\left(-h_i^+(z_{ij})\right)^{-1}\right]*y.
\end{align*}
The antipode map $S$ for $\textbf{A}^{\leq}_T$ is defined by:
\begin{align*}
    S\left(h^-_i(z)\right)&=\left(h^-_i(z)\right)^{-1}\\
    S(y)&=y*\left[\prod_{\substack{i\in I\\j\leq d_i}}\left(-h_i^-(z_{ij})\right)^{-1}\right].
\end{align*}
Using these antipode maps, $\textbf{A}^{\geq}_T$ and $\textbf{A}^{\leq}_T$ are Hopf algebras. The proof is the same as the one for the tripled quiver of the cyclic type $A$ quiver in \cite[Exercise IV.1]{n2}.

\subsection{The bialgebra pairing}
\label{pair}

\subsubsection{}
Let $(Q,W)$ be an arbitrary symmetric quiver with potential and let $T$ be a torus satisfying \eqref{torus3} and such that $R(d)^T$ is a point for any dimension vector $d\in\mathbb{N}^I$.
Consider the set $\mathcal{J}$ in $K_0^{T\times T(d)}(\text{pt})$ with functions $1-z_{ij}^{-1}z_{ik}$ for $i\in I$, $j\neq k$, and by $1-q^{-1}_\varepsilon z_{ij}^{-1}z_{i'j'}$ for $1\leq j\leq d_i$, $1\leq j'\leq d_{i'}$ and where $q_\varepsilon$ is a weight of $T$ corresponding to an edge $\varepsilon$ from $i$ to $i'$ of $Q$. None of the weights $q^{-1}_\varepsilon z_{ij}^{-1}z_{i'j'}$ are $1$ for $(i,j)=(i',j')$ by the assumption that $R(d)^T$ is a point.

Let $\iota_d: BT(d)\hookrightarrow \mathcal{Y}(d)$. We abuse notation and write 
\begin{equation*}\label{iota}
    \iota_d^*: K_0^T\left( D_{\text{sg}}(\X(d)_0)\right)\hookrightarrow K_0^T\left( D_{\text{sg}}(\mathcal{Y}(d)_0)\right)\to K_0^{T\times T(d)}(\text{pt}).
    \end{equation*}
Recall the definition of $\zeta_{ii'}$ from \eqref{zeta}. Define
\[
\widetilde{\zeta}_{ii'}\left(\frac{z_{ij}}{z_{i'j'}}\right)=
\begin{cases}
\zeta_{ii'}\left(\frac{z_{ij}}{z_{i'j'}}\right)\text{ if }(i,j)\neq (i',j')\\
\left(1-q_1^{-1}\right)\cdots \left(1-q_{\varepsilon(i,i)}^{-1}\right)\text{ otherwise}.
\end{cases}
\]
For $y$ in $K_0^T(D_{\text{sg}}(\X(d)_0))$, define
\[\Psi(y):=\frac{\iota_d^*(y)}{
\prod_{\substack{i,i'\in I\\j,j'}}\widetilde{\zeta}_{ii'}\left(\frac{z_{ij}}{z_{i'j'}}\right)}.\]
Define the pairing
\begin{multline*}
    \langle\,,\,\rangle: K_0^T(D_{\text{sg}}(\X(d)_0))\otimes K_0^T(D_{\text{sg}}(\X(d)_0))\xrightarrow{\otimes} K_0^T(D_{\text{sg}}(\X(d)_0))\to\\ K_0^T(D_{\text{sg}}(\X(d)_0))_\mathcal{J}\xrightarrow{\Psi}K_0^T(BT(d))_\mathcal{J}. 
\end{multline*}

\subsubsection{}\label{tripled}
Let $Q$ be an arbitrary quiver, let $\left(\widetilde{Q}, \widetilde{W}\right)$ be its tripled quiver, and let $T$ be a torus satisfying \eqref{torusT}. Then $\widetilde{R(d)}^T$ is a point. Let $\mathbf{F}=\text{Frac}\,K_0^{T}(\text{pt})$. 
Consider a new variable $p$. For $i\in I$, define 
\[\mu_i(z)=\frac{1-p^2z^{-1}}{1-q^2z^{-1}}.\]
Define the map 
\[\Phi: \text{image}\,(\Psi)_\mathbb{F}\to \mathbb{F} 
\] by the formula:
\[\Phi\left(\frac{y}{
\prod_{\substack{i,i'\in I\\j,j'}}\widetilde{\zeta}_{ii'}\left(\frac{z_{ij}}{z_{i'j'}}\right)}
\right)=\frac{1}{d!}
\int^{|q|<1<|p|}_{|z_{ij}|=1}
\frac{\iota_d^*(y)}{\prod_{\substack{i,i'\in I\\j,j'}}
\widetilde{\zeta}_{ii'}\left(\frac{z_{ij}}{z_{i'j'}}\right)
\prod_{\substack{i\in I\\j,k}}\mu_i\left(\frac{z_{ij}}{z_{ik}}\right)
}
\prod_{\substack{i\in I\\ j\leq d_i}} Dz_{ij}
\bigg|_{p\to q}
.\]
In the above, $Dz=\frac{dz}{2\pi iz}$. To evaluate $\Phi$, we first compute the integral in the region indicated
also assuming that $|q_e|=1$ for $e\in E$, and then we take the limit $p\to q$. The integral is equal to the sum of residues of the integrand in the region $|q|<1<|p|$, $|z_{ij}|=|q_e|=1$ for $i\in I$, $j\leq d_i$, $e\in E$.

 Define the pairing
 \begin{align*}
     (\,,\,):K_0^T\left(D_{\text{sg}}\left(\widetilde{\X(d)}_0\right)\right)_{\mathbb{F}}\otimes_{\mathbb{F}} K_0^T\left(D_{\text{sg}}\left(\widetilde{\X(d)}_0\right)\right)_{\mathbb{F}}&\to\mathbb{F}\\
     (x,y)&=\Phi\,\langle x,y\rangle.
 \end{align*}
Extend the pairing $(\,,\,)$ by the formulas:
\begin{align*}
    (\,,\,):\textbf{A}^{\leq}_{T,\mathbb{F}}\otimes_\mathbb{F} \textbf{A}^{\geq}_{T,\mathbb{F}}&\to \mathbb{F}\\
    (h_{i'}^-(z), h_i^+(w))&=\tau_{i'i}\left(\frac{w}{z}\right),\\
    (h_{i'}^-(z), y)&=0,\\
    (y, h_i^+(w))&=0
\end{align*}
for any element $y\in \text{KHA}^T\left(\widetilde{Q},\widetilde{W}\right)_\mathbb{F}$.

\textbf{Remark.}
By the localization theorem \cite[Theorem 2.5]{P}, there is an isomorphism
\[K_0^T\left( D_{\text{sg}}(\mathcal{Y}(d)_0)\right)_{\mathcal{I}}
\xrightarrow{\sim} K_0^{T\times T(d)}(\text{pt})_{\mathcal{I}}.\]
Using a shuffle description of the various maps involved, $\Psi$ can be thought as an inverse of the map in \eqref{iota}. It would be interesting to prove Theorem \ref{thm2} using such a geometric description of $\Phi$ and adjunction of the various functors used in the definitions of $m$ and $\Delta$.

\subsection{Proof of Theorem \ref{thm2}}

By a theorem of Negu\c{t}, the algebra $\text{KHA}_T\left(\widetilde{Q},\widetilde{W}\right)_{\mathbb{F}}$ is spherically generated, i.e. it is generated by elements in dimension $\varepsilon_i\in \mathbb{N}^I$, where $\varepsilon_i$ for $i\in I$ is a vector with $i$ component equal to $1$ and all other components are equal to zero, see \cite[Theorem 1.2]{n3}. 
We check that $(x, y*y')=(\Delta^{\text{op}}(x), y\otimes y')$ for $x, y, y'$ products of elements from such dimensions $\varepsilon_i$. Assume that $\deg(y)=a$, $\deg(y')=b$, $\deg(x)=a+b=d$. 

For elements $x, y, y'\in \text{KHA}_T\left(\widetilde{Q},\widetilde{W}\right)_\mathbb{F}$, the statement follows from the argument in \cite[Exercise IV.2]{n2}, which we now explain. Let $A_i=\{1,\cdots, a_i\}$ and $B_i=\{a_i+1,\cdots, d_i\}$. 
Let $\textbf{E}$ be the set of indices
\[\textbf{E}=\{(i,i',j,j')|\,i,i'\in I, (j,j')\in (A_i,A_{i'}), (B_i,B_{i'}), (B_i,A_{i'})\}.\]
We will use the shorthand notations $z_A$ for $z_{ij}$ with $i\in I$ and $j\in A_i$ and $z_B$ for $z_{ij}$ with $i\in I$ and $j\in B_i$.

Using Proposition \ref{p2}, part (a), and the fact that the region we integrate in is symmetric in the all variables $z_{ij}$, we obtain that:
\[(x,y*y')=\frac{1}{a!b!}
\int^{|q|<1<|p|}_{|z_A|=|z_B|=1}
\frac{\iota^*_d(x)\iota_a^*(y)\iota_b^*(y')}{\prod_{\textbf{E}}\widetilde{\zeta}_{ii'}\left(\frac{z_{ij}}{z_{i'j'}}\right)\prod_{\substack{i\in I\\j,k}}\mu_i\left(\frac{z_{ij}}{z_{ik}}\right)}
\prod_{\substack{i\in I\\ j\in A_i}} Dz_{ij} \prod_{\substack{i\in I\\ j\in B_i}} Dz_{ij}.\]
Call the integrand in the above formula $u$. Recall from the Subsection \ref{tripled} that the pairing between an element in $\text{KHA}_T\left(\widetilde{Q},\widetilde{W}\right)_\mathbb{F}$ and a $h^{\pm}_{i,n}$ is zero.
By the formula in \eqref{coprneg}, we have that:
\begin{align*}
    (\Delta^{\text{op}}(x), y\otimes y')&=
    (\Delta(x), y'\otimes y)\\
&=\frac{1}{a!b!}
\int^{|q|<1<|p|}_{|z_B|\ll|z_A|}
\frac{\iota^*_d(x)\iota_b^*(y')\iota_a^*(y)}{\prod_{\textbf{E}}\widetilde{\zeta}_{ii'}\left(\frac{z_{ij}}{z_{i'j'}}\right)\prod_{\substack{i\in I\\j,k}}\mu_i\left(\frac{z_{ij}}{z_{ik}}\right)}
\prod_{\substack{i\in I\\ j\in A_i}} Dz_{ij} \prod_{\substack{i\in I\\ j\in B_i}} Dz_{ij}\\
&=\frac{1}{a!b!}
\int^{|q|<1<|p|}_{|z_A|\gg|z_B|}
\frac{\iota^*_d(x)\iota_a^*(y)\iota_b^*(y')}{\prod_{\textbf{E}}\widetilde{\zeta}_{ii'}\left(\frac{z_{ij}}{z_{i'j'}}\right)\prod_{\substack{i\in I\\j,k}}\mu_i\left(\frac{z_{ij}}{z_{ik}}\right)}
\prod_{\substack{i\in I\\ j\in A_i}} Dz_{ij} \prod_{\substack{i\in I\\ j\in B_i}} Dz_{ij}.
\end{align*}
The last equality holds because $\iota_d^*(x)$ is $\mathfrak{S}_d$-symmetric. 
It suffices to show that
\begin{equation}\label{int2}
\int^{|q|<1<|p|}_{|z_A|=|z_B|=1}
u
\prod_{\substack{i\in I\\ j\in A_i}} Dz_{ij} \prod_{\substack{i\in I\\ j\in B_i}} Dz_{ij}
\bigg|_{p\to q}=
\int^{|q|<1<|p|}_{|z_A|\gg|z_B|}
u
\prod_{\substack{i\in I\\ j\in A_i}} Dz_{ij} \prod_{\substack{i\in I\\ j\in B_i}} Dz_{ij}
\bigg|_{p\to q}.
\end{equation}
The integrand $u$ has poles involving $z_A$ and $z_B$ in the left hand side region at 
\[z_A=qq_\varepsilon^{\pm 1}z_B,\quad z_{ij}=p^{-2}z_{ik},\quad z_{ij}=p^2z_{ik},\]
where $\varepsilon$ in an edge of $Q$, $j\in A_i$, $k\in B_i$.
There is a factor $z_{ij}=q^2z_{ik}$ in the numerator for $i\in A_i$, $j\in B_i$, and thus any residue obtained at $z_{ij}=p^2z_{ik}$ vanishes after setting $p\to q$. There are no other poles of $u$ involving $z_A$ and $z_B$ as we move to $|z_A|\gg |z_B|$. Similarly, we can also replace the factors $\mu_i\left(\frac{z_A}{z_B}\right)$ by $1$ as the two corresponding integrals coincide after passing to $p\to q$. The equality in \eqref{int2} thus follows. The statement $(x*x',y)=(x\otimes x', \Delta(y))$ follows similarly.
\\

The pairing $(\,,\,)$ extends to $\textbf{A}^{\leq}_{T,\mathbb{F}}\otimes_{\mathbb{F}} \textbf{A}^{\geq}_{T,\mathbb{F}}$. We show that
\begin{equation}\label{op}
\left(
\Delta^{\text{op}}(x), h^+_{i}(w)\otimes y\prod_{\substack{i'\in I\\j'\leq d_{i'}}}\tau_{ii'}\left(\frac{z_{i'j'}}{w}\right)
\right)
=
\left(\Delta^{\text{op}}(x), y\otimes h_{i}^+(w)\right)
.\end{equation}
Indeed, $\Delta^{\text{op}}(x)=x\otimes 1+\cdots+\prod_{\substack{i'\in I\\j'\leq d_{i'}}}h^-_{i'}(z_{i'j'})\otimes x$, where the middle terms do not contribute to the pairing, and thus the left hand side is equal to
\begin{align*}
\prod_{\substack{i'\in I\\j'\leq d_{i'}}}
\left(h^-_{i'}(z_{i'j'}), h^+_{i}(w)\right)
(x,y)
\prod_{\substack{i'\in I\\j'\leq d_{i'}}}\tau_{ii'}\left(\frac{z_{i'j'}}{w}\right)&=
\prod_{\substack{i'\in I\\j'\leq d_{i'}}}\tau_{i'i}\left(\frac{w}{z_{i'j'}}\right)(x,y)\prod_{\substack{i'\in I\\j'\leq d_{i'}}}\tau_{ii'}\left(\frac{z_{i'j'}}{w}\right)\\
&=(x,y).
\end{align*}
The right hand side is $(x,y)\left(1, h_{i'}(w)\right)=(x,y)$, so the equality in \eqref{op} holds. Similarly, we have that
\begin{equation*}
\left(
h^-_{i}(w)\otimes y\prod_{\substack{i'\in I\\j'\leq d_{i'}}}\tau_{ii'}\left(\frac{w}{z_{i'j'}}\right), \Delta(x)
\right)
=
\left(y\otimes h_{i}^-(w), \Delta(x)\right)
.\end{equation*}

Finally, let $x$ be a non-zero element in the image of $\Psi$. There exists a monomial $u$ such that $\Phi(xu)\neq 0$. For $y$ be non-zero in $K_0^T\left(D_{\text{sg}}\left(\widetilde{\X(d)}_0\right)\right)_{\mathbb{F}}$ and non-zero monomial $u$, we have that $\langle y,u\rangle\neq 0$. Thus there is a monomial $u$ such that $(y,u)=\Phi\langle y,u\rangle\neq 0$.

\subsection{Examples.} 
\subsubsection{}\label{ss1}
Let $J$ be a Jordan quiver and let $\mathbb{C}^*$ act on representations of $J$ by scaling the linear map. 
By 
\cite{Ts1}, \cite{VV}, $\text{KHA}_{\mathbb{C}^*}(J,0)\cong U^{>}_q\left(L\mathfrak{sl}_2\right)$ and thus $\textbf{A}^{\geq}_{\mathbb{C}^*}\cong U_q\left(L\mathfrak{b}\right)\cong U^{\geq}_q\left(L\mathfrak{sl}_2\right)$, where $\mathfrak{b}$ is a Borel subalgebra of $\mathfrak{sl}_2$.
For $a\in\mathbb{Z}$,
denote by $e_a:=z^a$ generators of $\text{KHA}_{\mathbb{C}^*}(J,0)\subset \textbf{A}^{\geq}_{\mathbb{C}^*}$ and by $f_a:=z^a$ generators of $\text{KHA}_{\mathbb{C}^*}(J,0)^{\text{op}}\subset \textbf{A}^{\leq}_{\mathbb{C}^*}$. 
The formulas for the coproduct are:
\begin{align*}
    \Delta(e_a)&=e_a\otimes 1+\sum_{n\geq 0}h^+_n\otimes e_{a-n}\\
    \Delta(f_a)&=1\otimes f_a+\sum_{n\geq 0}f_{a-n}\otimes h^-_n.
\end{align*}
We next compute the bialgebra pairing on $\textbf{A}^{\leq}_{\mathbb{C}^*}\otimes \textbf{A}^{\geq}_{\mathbb{C}^*}$. For the relations in $\textbf{A}^{\leq}_{\mathbb{C}^*}$, see \cite[Relation 6.17, Example A.27]{S}, in particular $[H_n, e_a]$ is equal to $e_{a+n}$ up to a constant, where $H_n$ is explicitly defined in terms of $h^+_n$, see loc. cit.
By definition, we have that
\begin{align*}
\Big(h^-(z), h^+(w)\Big)&=\tau\left(\frac{w}{z}\right)=\frac{qw-z}{w-qz}\\
(h^-_b, e_a)&=0\\
(f_a, h^+_b)&=0.
\end{align*}
Finally, we compute 
\[(f_m, e_n)=\int^{|p|<1<|q|}_{|z|=1}\frac{z^{n+m}}{1-q^2}\frac{1-q^2}{1-p^2}Dz\bigg|_{p\to q}=\frac{1}{1-q^2}\int _{|z|=1}z^{n+m}Dz=\frac{\delta_{0,n+m}}{1-q^2}.\]
Replacing $q$ with $q^{-1}$ and up to a factor in $(f_m, e_n)$, we obtain the bialgebra pairing on $U^{\leq}_q\left(L\mathfrak{sl}_2\right)\otimes U^{\geq}_q\left(L\mathfrak{sl}_2\right)$ in \cite[Proposition 2.34]{Ts2}. Thus $\textbf{A}_{\mathbb{C}^*}\cong U_q\left(L\mathfrak{sl}_2\right)$.

\subsubsection{}\label{ss2} Consider the quiver 
$Q_3$ with one vertex and three loops $\ell_1,\ell_2,\ell_3$ with potential $W_3=\ell_1\ell_2\ell_3-\ell_1\ell_3\ell_2$. It is the tripled quiver of the Jordan quiver $J$. Let $T$ be the two dimensional torus acting with weights $t$ on the edge $\ell_1$, $t^{-1}$ on $\ell_2$, and with weights $q$ on $\ell_1$ and $\ell_2$ and $q^{-2}$ on $\ell_3$.  The spherical subalgebra of
$K_0^T\left(D_{\text{sg}}\left(\widetilde{\X(d)}_0\right)\right)_{\mathbb{F}}$ has been studied by Negu\c{t} \cite{n1}, \cite{n2}, and is isomorphic to $U_{q, t}^{>}\left(\widehat{\widehat{\mathfrak{gl}_1}}\right)$. Consider the generators $e_{n}:=z^n$ of $\text{KHA}_T(Q_3,W_3)_{\mathbb{F}}$ and $f_n:=z^n$ of $\text{KHA}_T(Q_3,W_3)_{\mathbb{F}}^{\text{op}}$. 
Then 
\[\Delta(e_{n})=e_{n}\otimes 1+\sum_{i\geq 0}h^+_i\otimes e_{n-i}.\]
The coproduct is the same as the one in \cite[Subsection 4.6]{n1}. For the relations in the extended algebra, there are explicitly defined $H_n$ in terms of $h^+_n$, see \cite[Relations (3.4) and (4.8)]{n1}, such that $[H_i, e_n]=e_{i+n}$.
The pairing on generators of $\text{KHA}_T(Q_3,W_3)_{\mathbb{F}}$ is:
\begin{align*}
(f_{m}, e_n)&=\int^{|p|<1<|q|}_{|z|=1}\frac{z^{n+m}}{(1-q^{-1}t)(1-q^{-1}t^{-1})(1-q^2)}Dz\bigg|_{p\to q}
\\&=\frac{\delta_{0,n+m}}{(1-q^{-1}t)(1-q^{-1}t^{-1})(1-q^2)}.\end{align*}
This is the same as the pairing in \cite[Subsection 4.6]{n1}.


\end{document}